\documentclass[final,3p,times,twocolumn]{elsarticle}

\usepackage{amsmath, amssymb, amsthm, bm, nccmath}
\usepackage{algorithmic}
\usepackage{graphicx}
\usepackage{textcomp}
\usepackage{xcolor}
\biboptions{sort&compress}

\usepackage{optidef}
\def\matt#1{\begin{bmatrix}#1\end{bmatrix}}

\newtheorem{theorem}{Theorem}
\newtheorem{corollary}{Corollary}
\newtheorem{lemma}{Lemma}
\newtheorem{assumption}{Assumption}

\newtheorem{remark}{Remark}
\newtheorem{definition}{Definition}

\journal{Systems and Control Letters}

\begin{document}
\begin{frontmatter}

\title{
Distributed Optimization under Edge Agreements: A Continuous-Time Algorithm
}

\author[inst1]{Zehui Lu\corref{cor1}}
\ead{lu846@purdue.edu}
\affiliation[inst1]{organization={School of Aeronautics and Astronautics, Purdue University},
addressline={701 W Stadium Ave.},
postcode={47907},
city={West Lafayette, IN},
country={USA}}
\author[inst1]{Shaoshuai Mou}
\ead{mous@purdue.edu}
\cortext[cor1]{Corresponding author}

\begin{abstract}
Generalized from the concept of consensus, this paper considers a group of edge agreements, i.e. constraints defined for neighboring agents, in which each pair of neighboring agents is required to satisfy one edge agreement constraint. Edge agreements are defined locally to allow more flexibility than a global consensus. This work formulates a multi-agent optimization problem under edge agreements and proposes a continuous-time distributed algorithm to solve it. Both analytical proof and numerical examples are provided to validate the effectiveness of the proposed algorithm.
\end{abstract}


\end{frontmatter}

\section{Introduction} \label{section:introduction}

There has recently been a significant amount of research interest in distributed algorithms for networked multi-agent systems (MAS), which focus on accomplishing global objectives via only local coordination through the network. To guarantee all agents in a networked MAS work as a cohesive whole, the concept of consensus that requires all agents to reach an agreement regarding a certain quantity \cite{cao2008reaching,cao2008agreeing}, has naturally arisen and started serving as a basis to develop distributed algorithms for MAS.
Besides multi-agent computation \cite{mou2015distributed, wang2019distributed, wang2019scalable,JXSB20TAC,ASB23CSL}, multi-agent optimal control \cite{movric2013cooperative, lu2022cooperative}, and multi-agent formation control \cite{sun2016exponential, mou2015undirected, chen2019controlling}, consensus-based distributed optimization for MAS \cite{nedic2009distributed,nedic2010constrained,gharesifard2013distributed,kia2015distributed,wang2022consensus,qu2019accelerated, tang2020distributed, tang2021communication,duchi2011dual, terelius2011decentralized, chang2014distributed} has recently attracted a lot of research attention, for which each agent knows one objective function and a local constraint, and the goal is to minimize the sum of local objective functions subject to all local constraints. Along this direction, primal-dual methods \cite{duchi2011dual, terelius2011decentralized, chang2014distributed} update and communicate both primal and dual variables to optimize the global objective function via dual decomposition. Based on augmented Lagrangian, the authors of \cite{chang2016proximal, falsone2017dual, aybat2017distributed} propose a distributed alternating direction method of multipliers (ADMM), which decomposes dual variables by primal decomposition.
Another family of distributed optimization methods comes from the integration of consensus with gradient descent method and projection operators to deal with objective functions and local constraints, respectively, such as \cite{nedic2009distributed, nedic2010constrained} with asymptotic convergence under a diminishing step size (e.g., $1/t$).
To further eliminate such diminishing step size and improve the convergence rate, the work \cite{gharesifard2013distributed, liu2015second, qiu2016distributed, zeng2016distributed} introduce a Lagrangian dual vector for consensus errors, and the results in  \cite{kia2015distributed, wang2022consensus} employ the integral of consensus error, which achieve an exponential convergence rate and are thus robust against bounded disturbance.
Some work \cite{qu2019accelerated, tang2020distributed, tang2021communication} focus on a faster convergence rate, non-convex and non-smooth optimization, and communication-efficient distributed algorithm, respectively.

With so many distributed algorithms developed based on consensus as mentioned above, they are designed only for scenarios when all agents need to reach the same value regarding a specific quantity.
Recognition of this has motivated the goal of this paper to investigate coordination among agents beyond consensus. Different from consensus, which enforces a global constraint to all agents in the whole MAS, we in this paper consider a group of \emph{edge agreements}, i.e. constraints defined for nearby neighboring agents, with each pair of neighboring agents corresponding to one such constraint.
Defined locally to allow more flexibility than a global consensus, edge agreements can deal with cases when heterogeneous coordination among neighbor agents is needed, such as multi-agent formation control in which each pair of neighbor agents needs to maintain a different distance constant.
The global consensus can be treated as a special case of edge agreements, as indicated in \cite{rai2023edge}, which has investigated distributed algorithms to achieve edge agreements for multi-agent systems in discrete time.

Motivated by the context above, this paper will develop a distributed algorithm for multi-agent optimization under edge agreements, for which all agents aim to minimize the sum of each agent's local objective function. In contrast, each pair of neighbor agents satisfies a prescribed edge agreement constraint.
Potential applications for multi-agent optimization under edge agreements include distributed model predictive control (MPC) \cite{stewart2010cooperative} and applying such techniques to battery network energy management \cite{fang2016cooperative}, where the equality constraint of linear dynamics for each agent and the consensus on control among different agents can be unified as edge-agreement constraints. The contribution of this paper is a distributed algorithm for multi-agent optimization under edge agreements in continuous time.

This paper is organized as follows: A mathematical problem formulation is given in Section \ref{section:problem_formulation} followed by some preliminary discussions in Section \ref{section:prelim}. A proposed distributed algorithm and some theoretical analysis are developed in Section \ref{section:method}, which deal with multi-agent optimization under edge agreements. Simulations on two examples of multi-agent formation control are given in Section \ref{section:example}.


\textbf{\emph{Notations.}}  Let $|\mathcal{N}|$ denote the cardinality of a set $\mathcal{N}$. Let $(\cdot)^\prime$ denote the Hermitian transpose. Let $||\cdot||$ denote the Euclidean norm.
Let $\text{col}\{ \boldsymbol{v}_1, \cdots, \boldsymbol{v}_a \}$ denote a column stack of elements $\boldsymbol{v}_1, \cdots, \boldsymbol{v}_a $, which may be scalars, vectors or matrices,
i.e. $\text{col}\{ \boldsymbol{v}_1, \cdots, \boldsymbol{v}_a \} \triangleq {\matt{{\boldsymbol{v}_1}^{\prime} & \cdots & {\boldsymbol{v}_a}^{\prime}}}^{\prime}$.
Let $\otimes$ denote the Kronecker product.
Let $I_n \in \mathbb{R}^{n \times n}$ denote an identity matrix in $\mathbb{R}^{n \times n}$.

\section{Problem Formulation} \label{section:problem_formulation}

Consider a multi-agent system consisting of $m$ agents labeled as $\mathcal{V}=\{1,\cdots,m\}$. Each agent $i$ is modeled by a single integrator, can control a state $\boldsymbol{x}_i \in \mathbb{R}^n$ and communicate bidirectionally within its nearby neighbors denoted by $\mathcal{N}_i$. Here, we assume $i \notin \mathcal{N}_i$. Let $\mathbb{G} = \{ \mathcal{V}, \mathcal{E} \}$ denote the undirected graph such that an undirected edge $(i,j) \in \mathcal{E} $ if and only if agent $i$ and agent $j$ are neighbors. Let $\bar{m} \triangleq |\mathcal{E}|$ the number of edges in $\mathbb{G}$. Suppose each agent $i$ only knows its private objective function $\boldsymbol{f}_i: \mathbb{R}^n \mapsto \mathbb{R}$.

The \textbf{problem of interest} to develop an update rule for each agent $i$ to update $\boldsymbol{x}_i \in \mathbb{R}^n$ such that each $\boldsymbol{x}_i$ converges to a constant vector, which minimizes the global sum of individual objective functions $\sum_{i=1}^{m} \boldsymbol{f}_i$ and satisfies the following edge agreements
\begin{equation} \label{eq:edge_agreement_require}
A_{ij}(\boldsymbol{x}_i-\boldsymbol{x}_j)=\boldsymbol{b}_{ij}, \ \forall (i,j) \in \mathcal{E}.
\end{equation}
That is,
\begin{equation} \label{problem_interest}
\begin{aligned}
{ \underset {{\{\boldsymbol{x}_i\}}_{i=1}^m} {\min} }  \quad & \textstyle\sum_{i=1}^{m} \boldsymbol{f}_i(\boldsymbol{x}_i)\\
\textrm{s.t.} \quad & A_{ij}(\boldsymbol{x}_i-\boldsymbol{x}_j)=\boldsymbol{b}_{ij}, \ \forall (i,j) \in \mathcal{E}.    \\
\end{aligned}
\end{equation}
Here, $A_{ij} \in \mathbb{R}^{d_{ij} \times n}$ and $\boldsymbol{b}_{ij} \in \mathbb{R}^{d_{ij}}$ are constant matrices, and privately known to agent $i$;
$d_{ij}$ is the dimension of edge agreement associated with edge $(i,j)$;
$\boldsymbol{f}_i: \mathbb{R}^n \mapsto \mathbb{R}$ is continuously differentiable and convex, which is privately known to agent $i$. Denote $\boldsymbol{f}(\boldsymbol{x}) \triangleq \sum_{i=1}^{m} \boldsymbol{f}_i(\boldsymbol{x}_i)$, where $\boldsymbol{x} \triangleq \text{col}\{ \boldsymbol{x}_1, \cdots, \boldsymbol{x}_m \} \in \mathbb{R}^{mn}$, and $\boldsymbol{f}:\mathbb{R}^{mn} \mapsto \mathbb{R}$.
$\boldsymbol{f}$ is assumed to be convex in $\mathbb{R}^{mn}$.

\begin{remark}
If the objective function $\textstyle\sum_{i=1}^{m} \boldsymbol{f}_i(\boldsymbol{x}_i)$ is a constant, i.e. there is no objective function to be optimized, the problem \eqref{problem_interest} is reduced to an edge-agreement-only problem \eqref{eq:edge_agreement_require} for multi-agent systems, which has been solved by two discrete-time distributed algorithms in \cite{rai2023edge}.
Moreover, when $\boldsymbol{b}_{ij} = \boldsymbol{0}$, the consensus is one solution to the edge agreements \eqref{eq:edge_agreement_require}, although edge agreements do not necessarily imply consensus. The matrices $A_{ij}$ and vectors $\boldsymbol{b}_{ij}$ in the edge agreements can enable more flexible coordination among multi-agent systems by enforcing a linear constraint to each pair of neighboring agents, rather than a global requirement of consensus in a MAS. Such characterization of edge agreements can be employed to define a multi-agent formation as shown later in simulations.
\end{remark}

\section{Preliminary Results} \label{section:prelim}
This section introduces some assumptions and graphical notations.
\begin{assumption} \label{assum:existence}
\emph{(\textbf{Existence})} There exist at least one solution $\boldsymbol{x}^* \in \mathbb{R}^{mn}$, $\boldsymbol{x}^* \triangleq \emph{col}\{ \boldsymbol{x}^*_1, \cdots, \boldsymbol{x}^*_m \}$, to the problem \eqref{problem_interest} that satisfy the edge agreements.
\end{assumption}

Assumption \ref{assum:existence} guarantees that the problem \eqref{problem_interest} to be solved has solutions. In addition, we make the following assumptions regarding the edge agreement in \eqref{eq:edge_agreement_require}. Moreover, 
\begin{assumption} \label{assum:consistency} \emph{(\textbf{Consistency})}
Given the undirected graph $\mathbb{G} = \{ \mathcal{V}, \mathcal{E} \}$, the linear constraints for edge agreements \eqref{eq:edge_agreement_require} are consistent with each other for each pair of neighboring agents $(i,j) \in \mathcal{E}$, i.e.
\begin{equation*}
    A_{ij} = A_{ji}, \  \boldsymbol{b}_{ij} = - \boldsymbol{b}_{ji}.
\end{equation*}
\end{assumption}
The edge agreements \eqref{eq:edge_agreement_require} are equivalent to some constraints using projection matrices. Let $P_{ij} \in \mathbb{R}^{n \times n}$ denote a projection matrix and $\bar{\boldsymbol{b}}_{ij} \in \mathbb{R}^n$ denote as follows
\begin{equation} \label{eq:mat_define}
    P_{ij} = A_{ij}^{\prime} (A_{ij}A_{ij}^{\prime})^{-1}A_{ij}, \  \bar{\boldsymbol{b}}_{ij} = A_{ij}^{\prime} (A_{ij}A_{ij}^{\prime})^{-1} \boldsymbol{b}_{ij}.
\end{equation}
Let
\begin{subequations}
\begin{align}
\bar{P} &= \text{diag} \{ P_{i_1, j_1}, P_{i_2, j_2}, \cdots, P_{i_{\bar{m}},  j_{\bar{m}}} \} \in \mathbb{R}^{\bar{m}n \times \bar{m}n}, \label{eq:diag_mat_define} \\
\bar{\boldsymbol{b}} &= \text{col} \{ \bar{\boldsymbol{b}}_{i_1, j_1}, \cdots, \bar{\boldsymbol{b}}_{i_l, j_l}, \cdots, \bar{\boldsymbol{b}}_{i_{\bar{m}},  j_{\bar{m}}}  \}  \in \mathbb{R}^{\bar{m}n}, \label{eq:edge_b_define}
\end{align}
\end{subequations}
where $(i_l, j_l)$ denotes the $l$-th edge of the graph $\mathbb{G}$. By \eqref{eq:mat_define} and \eqref{eq:diag_mat_define}, note that $\bar{P}^2 = \bar{P}$ and $\bar{P}^{\prime} = \bar{P}$. Based on the definition of $P_{ij}$ and $\bar{\boldsymbol{b}}_{ij}$ in \eqref{eq:mat_define} and the fact that $P_{ij} \bar{\boldsymbol{b}}_{ij} = \bar{\boldsymbol{b}}_{ij}$, the linear constraints in \eqref{eq:edge_agreement_require} for edge agreements are equivalent to
\begin{equation} \label{eq:edge_agreement_each_after}
    P_{ij} (\boldsymbol{x}_i-\boldsymbol{x}_j-\bar{\boldsymbol{b}}_{ij}) = \boldsymbol{0}, \ \forall (i,j) \in \mathcal{E}.
\end{equation}

For the $m$-node-$\bar{m}$-edge undirected graph $\mathbb{G}$, one defines the oriented incidence matrix of $\mathbb{G}$ denoted by $H \in \mathbb{R}^{\bar{m} \times m}$ such that its entry at the $k$-th row and the $j$-th column is 1 if edge $k$ is an incoming edge to node $j$; -1 if edge $k$ is an outgoing edge to node $j$; and 0 elsewhere.
Note that for undirected graphs, the direction for each edge could be arbitrary, as long as it is consistent with each $\boldsymbol{b}_{ij}$, i.e. satisfy Assumption \ref{assum:consistency}.
Based on \eqref{eq:edge_agreement_each_after}, the definitions of $H$ and $\bar{P}$, and
\begin{equation} \label{eq:H_bar_define}
    \bar{H} \triangleq H \otimes I_n,
\end{equation} Let $\boldsymbol{x} \triangleq \{ \boldsymbol{x}_1, \cdots, \boldsymbol{x}_m  \} \in \mathbb{R}^{mn}$. Then by the definition of $\bar{H}$ and $\bar{P}$,
one has the following lemma:
\begin{lemma} \label{lemma:edge_agreement}
The edge agreements in \eqref{eq:edge_agreement_require} are equivalent to be
\begin{equation} \label{eq:edge_agreement_each_after1}
    \bar{P}(\bar{H} \boldsymbol{x}- \bar{\boldsymbol{b}}) = \boldsymbol{0},
\end{equation}
where $\bar{P}$, $\bar{\boldsymbol{b}}$, and $\bar{H} \in \mathbb{R}^{\bar{m}n \times mn}$ are as defined in \eqref{eq:diag_mat_define}, \eqref{eq:edge_b_define}, and \eqref{eq:H_bar_define}, respectively.
\end{lemma}

Before proceeding, we also make the following assumption about $\bar{H}$ and $\bar{P}$:
\begin{assumption} \label{assump:well_configured}
The graph $\mathbb{G}$ is connected and well-configured for edge agreements, i.e.
\begin{equation} \label{eq:well_configured}
    \ \emph{ker } \bar{H}^{\prime} \cap \emph{image } \bar{P} = \boldsymbol{0}.
\end{equation}
\end{assumption}

By the definition of kernel and image, $\text{ker } \bar{H}^{\prime} \triangleq \{ \bar{\boldsymbol{x}} \in \mathbb{R}^{\bar{m}n} \ | \ \bar{H}^{\prime} \bar{\boldsymbol{x}} = \boldsymbol{0} \}$ and $\text{image } \bar{P} \triangleq \{ \boldsymbol{y} \in \mathbb{R}^{\bar{m}n} \ | \  \boldsymbol{y} = \bar{P} \bar{\boldsymbol{x}}, \ \bar{\boldsymbol{x}} \in \mathbb{R}^{\bar{m}n} \}$. Since $\bar{P}$ is a diagonal matrix of projections, $\text{image } \bar{P} = \mathbb{R}^{\bar{m}n}$. Then $\emph{ker } \bar{H}^{\prime} \cap \emph{image } \bar{P} = \boldsymbol{0}$ indicates that
$\bar{H}^{\prime} \bar{P} \bar{\boldsymbol{x}} = \boldsymbol{0} \Rightarrow \bar{P} \bar{\boldsymbol{x}} = \boldsymbol{0}$, which further implies that 
$\bar{H}^{\prime}\bar{P}(\bar{H} \boldsymbol{x} - \bar{\boldsymbol{b}}) = \boldsymbol{0} \Rightarrow \bar{P}(\bar{H} \boldsymbol{x} - \bar{\boldsymbol{b}}) = \boldsymbol{0}$.

\section{Algorithm and Main Results} \label{section:method}

In this section, a distributed augmented Lagrangian multiplier algorithm in continuous time is proposed to solve the multi-agent optimization problem subject to edge-agreement constraints as in  \eqref{problem_interest}. 

\subsection{Proposed Distributed Algorithm}
Since the edge agreements \eqref{eq:edge_agreement_require} are equivalent to \eqref{eq:edge_agreement_each_after1} as indicated in Lemma \ref{lemma:edge_agreement}, solving the problem \eqref{problem_interest} is then equivalent to
\begin{equation} \label{eq:problem_equivalent}
\begin{aligned}
{ \underset {\boldsymbol{x} \in \mathbb{R}^{mn}} {\min} }  \quad & \boldsymbol{f}(\boldsymbol{x}) 
\quad \textrm{s.t.} \quad &  \bar{P}(\bar{H} \boldsymbol{x}- \bar{\boldsymbol{b}}) = \boldsymbol{0},
\end{aligned}
\end{equation}
where $\boldsymbol{x} = \text{col}\{ \boldsymbol{x}_1, \cdots, \boldsymbol{x}_m \} \in \mathbb{R}^{mn}$, and $\boldsymbol{f}:\mathbb{R}^{mn} \mapsto \mathbb{R}$ is $ \boldsymbol{f}(\boldsymbol{x}) 
 \triangleq\sum_{i=1}^{m} \boldsymbol{f}_i(\boldsymbol{x}_i) $.

To solve the constrained optimization in \eqref{eq:problem_equivalent}, we introduce the Lagrangian multiplier $\boldsymbol{\lambda} = \text{col} \{ \boldsymbol{\lambda}_1, \cdots, \boldsymbol{\lambda}_{m} \} \in \mathbb{R}^{mn}$ with $\boldsymbol{\lambda}_i \in \mathbb{R}^{n}$. Let
\begin{equation} \label{eq:lagrange_define}
\boldsymbol{F}(\boldsymbol{x}, \boldsymbol{\lambda}) = \boldsymbol{f}(\boldsymbol{x}) +  \boldsymbol{\lambda}^{\prime} \bar{H}^{\prime}\bar{P}(\bar{H} \boldsymbol{x} - \bar{\boldsymbol{b}}) + \frac{1}{2}||\bar{P}(\bar{H} \boldsymbol{x} - \bar{\boldsymbol{b}})||^2,
\end{equation}
Note that $\boldsymbol{F}:\mathbb{R}^{mn} \times \mathbb{R}^{mn} \mapsto \mathbb{R}$ is continuously differentiable since $\boldsymbol{f}$ is continuously differentiable. Since $\boldsymbol{F}$ is convex in its first argument and linear in the second argument, one naturally employs the following saddle-point dynamics, i.e., gradient descent in $\boldsymbol{x}$ and gradient ascent in $\boldsymbol{\lambda}$:
\begin{equation} \label{eq:algorithm_primal_dual_compact}
\begin{split}
\dot{\boldsymbol{x}}(t) =& -\nabla_{\boldsymbol{x}} \boldsymbol{F}(\boldsymbol{x}, \boldsymbol{\lambda}) = - \nabla \boldsymbol{f}(\boldsymbol{x}(t)) \\
&\  - \bar{H}^{\prime}\bar{P}\bar{H} \boldsymbol{\lambda}(t) - \bar{H}^{\prime}\bar{P}(\bar{H} \boldsymbol{x}(t) - \bar{\boldsymbol{b}}), \\
\dot{\boldsymbol{\lambda}}(t) =& \nabla_{\boldsymbol{\lambda}} \boldsymbol{F}(\boldsymbol{x}, \boldsymbol{\lambda}) = \bar{H}^{\prime}\bar{P}(\bar{H} \boldsymbol{x}(t) - \bar{\boldsymbol{b}}).
\end{split}
\end{equation} This leads to the following update at each agent $i$:
\begin{equation} \label{eq:algorithm_primal_dual_single}
\begin{split}
\dot{\boldsymbol{x}}_i(t) =& - \nabla \boldsymbol{f}_i(\boldsymbol{x}_i(t)) - \sum_{j \in \mathcal{N}_i} P_{ij}(\boldsymbol{\lambda}_i(t) - \boldsymbol{\lambda}_j(t)) \\
&- \sum_{j \in \mathcal{N}_i} P_{ij}(\boldsymbol{x}_i(t) - \boldsymbol{x}_j(t) - \bar{\boldsymbol{b}}_{ij}), \\
\dot{\boldsymbol{\lambda}}_i(t) =& \sum_{j \in \mathcal{N}_i} P_{ij}(\boldsymbol{x}_i(t) - \boldsymbol{x}_j(t) - \bar{\boldsymbol{b}}_{ij}),
\end{split}
\end{equation}
where $P_{ij}$ and $\bar{\boldsymbol{b}}_{ij}$ are as defined in \eqref{eq:mat_define}.
Note right away that the update \eqref{eq:algorithm_primal_dual_single} is distributed since it only requires local information available to each agent's neighbors. Moreover, one has the following theorem to prove global convergence.

\begin{theorem}[\textbf{Multi-Agent Optimization under Edge Agreement}]\label{theorem:saddle_point_theorem}
Let Assumption \ref{assum:existence} - \ref{assump:well_configured} hold. Suppose $\boldsymbol{f} = \sum_{i = 1}^m \boldsymbol{f}_i$ is strictly convex and each $\boldsymbol{f}_i$ is convex for all $i \in \mathcal{V}$. Then for any given $\boldsymbol{x}(0)$, the distributed update \eqref{eq:algorithm_primal_dual_single} drives each $\boldsymbol{x}_i(t)$ to converge asymptotically fast to a constant $\boldsymbol{x}_i^*$, $i\in \mathcal{V}$ that solve the problem  in \eqref{problem_interest}. 

\end{theorem}

\begin{remark}
If the objective function is a constant, the proposed distributed algorithm \eqref{eq:algorithm_primal_dual_single} is reduced to the same one without the term $-\nabla \boldsymbol{f}_i(\boldsymbol{x}_i(t))$.
\end{remark}

To prove Theorem \ref{theorem:saddle_point_theorem}, we first introduce the following definitions.
\begin{definition} [\textbf{Saddle Point}] \label{definition:saddle_point}
A point $(\boldsymbol{x}^*, \boldsymbol{\lambda}^*) \in \mathbb{R}^n \times \mathbb{R}^m$ is a local min-max saddle point of a continuously differentiable function $\boldsymbol{F}:\mathbb{R}^n \times \mathbb{R}^m \mapsto \mathbb{R}$ if there exist open neighborhoods $\mathcal{U}_{\boldsymbol{x}^*} \subset \mathbb{R}^n$ of $\boldsymbol{x}^*$ and $\mathcal{U}_{\boldsymbol{\lambda}^*} \subset \mathbb{R}^m$ of $\boldsymbol{\lambda}^*$ such that
\begin{equation}
\boldsymbol{F}(\boldsymbol{x}^*, \boldsymbol{\lambda}) \leq \boldsymbol{F}(\boldsymbol{x}^*, \boldsymbol{\lambda}^*) \leq \boldsymbol{F}(\boldsymbol{x}, \boldsymbol{\lambda}^*)
\end{equation}
for all $\boldsymbol{\lambda}^* \in \mathcal{U}_{\boldsymbol{\lambda}^*}$ and $\boldsymbol{x}^* \in \mathcal{U}_{\boldsymbol{x}^*}$. The point $(\boldsymbol{x}^*, \boldsymbol{\lambda}^*)$ is a global min-max saddle point of $\boldsymbol{F}$ if $\mathcal{U}_{\boldsymbol{x}^*} = \mathbb{R}^n$ and $\mathcal{U}_{\boldsymbol{\lambda}^*} = \mathbb{R}^m$.
\end{definition}

\begin{definition} [\textbf{Convex-concavity}] \label{definition:convex-concave}
A function $\boldsymbol{F}:\mathbb{R}^n \times \mathbb{R}^m \mapsto \mathbb{R}$ is globally convex-concave in $\mathbb{R}^n \times \mathbb{R}^m$ if for all $(\bar{\boldsymbol{x}}, \bar{\boldsymbol{\lambda}}) \in \mathbb{R}^n \times \mathbb{R}^m$, the functions $\boldsymbol{x} \mapsto \boldsymbol{F}(\boldsymbol{x}, \bar{\boldsymbol{\lambda}})$ and $\boldsymbol{\lambda} \mapsto \boldsymbol{F}(\bar{\boldsymbol{x}}, \boldsymbol{\lambda})$ are convex over $\mathbb{R}^n$ and concave over $\mathbb{R}^m$, respectively.
$\boldsymbol{F}$ is globally strictly convex-concave if it is globally convex-concave and for any $(\bar{\boldsymbol{x}}, \bar{\boldsymbol{\lambda}}) \in \mathbb{R}^n \times \mathbb{R}^m$, \textbf{\emph{either}} $\boldsymbol{x} \mapsto \boldsymbol{F}(\boldsymbol{x}, \bar{\boldsymbol{\lambda}})$ is strictly convex, \textbf{\emph{or}} $\boldsymbol{\lambda} \mapsto \boldsymbol{F}(\bar{\boldsymbol{x}}, \boldsymbol{\lambda})$ is strictly concave.
\end{definition}

Next, let $\text{Saddle}(\boldsymbol{F})$ denote the set of saddle points of $\boldsymbol{F}$, and one has the following lemma.

\begin{lemma} \label{lemma:saddle_point}
Every equilibrium of the dynamics \eqref{eq:algorithm_primal_dual_compact} is a global saddle point of the augmented Lagrangian $\boldsymbol{F}$. 
\end{lemma}
\begin{proof} 
By Assumption \ref{assum:existence}, $\text{Saddle}(\boldsymbol{F}) \neq \emptyset$.
The set of the equilibrium of the dynamics \eqref{eq:algorithm_primal_dual_compact} or \eqref{eq:algorithm_primal_dual_single} satisfies
\begin{equation} \label{eq:saddle_point_derive}
\begin{split}
\dot{\boldsymbol{x}}(t) &= - \nabla \boldsymbol{f}(\boldsymbol{x}(t)) - \bar{H}^{\prime}\bar{P}\bar{H} \boldsymbol{\lambda}(t) - \bar{H}^{\prime}\bar{P}(\bar{H} \boldsymbol{x}(t) - \bar{\boldsymbol{b}}) = \boldsymbol{0}, \\
\dot{\boldsymbol{\lambda}}(t) &= \bar{H}^{\prime}\bar{P}(\bar{H} \boldsymbol{x}(t) - \bar{\boldsymbol{b}}) = \boldsymbol{0}.
\end{split}
\end{equation}
Denote the set of the dynamics' equilibrium as $\text{Equilibrium}(\boldsymbol{F})$. Then \eqref{eq:saddle_point_derive} yields
\begin{equation} \label{eq:eq_F_define}
\begin{split}
&\text{Equilibrium}(\boldsymbol{F}) \equiv \Big\{ (\bar{\boldsymbol{x}}, \bar{\boldsymbol{\lambda}}) \ \Big| \ \nabla \boldsymbol{f}(\bar{\boldsymbol{x}}) + \bar{H}^{\prime}\bar{P}\bar{H} \bar{\boldsymbol{\lambda}} = \boldsymbol{0} \\
&\text{ and }\bar{H}^{\prime}\bar{P}(\bar{H} \bar{\boldsymbol{x}} - \bar{\boldsymbol{b}}) = \boldsymbol{0} \Big\}.
\end{split}
\end{equation}
Following Definition \ref{definition:saddle_point} and the convexity of $\boldsymbol{f}_i$, one can prove that every equilibrium of saddle-point dynamics \eqref{eq:algorithm_primal_dual_compact} is a global saddle point. To see this, for every $(\bar{\boldsymbol{x}}, \bar{\boldsymbol{\lambda}}) \in \text{Equilibrium}(\boldsymbol{F})$,
\begin{equation}
\begin{split}
&\boldsymbol{F}(\bar{\boldsymbol{x}}, \boldsymbol{\lambda}) - \boldsymbol{F}(\bar{\boldsymbol{x}}, \bar{\boldsymbol{\lambda}}) = \\
&\boldsymbol{f}(\bar{\boldsymbol{x}}) + {\boldsymbol{\lambda}}^{\prime} \bar{H}^{\prime}\bar{P}(\bar{H} \bar{\boldsymbol{x}} - \bar{\boldsymbol{b}}) + \frac{1}{2}||\bar{P}(\bar{H} \bar{\boldsymbol{x}} - \bar{\boldsymbol{b}})||^2 \\
& -\boldsymbol{f}(\bar{\boldsymbol{x}}) - {\bar{\boldsymbol{\lambda}}}^{\prime} \bar{H}^{\prime}\bar{P}(\bar{H} \bar{\boldsymbol{x}} - \bar{\boldsymbol{b}}) - \frac{1}{2}||\bar{P}(\bar{H} \bar{\boldsymbol{x}} - \bar{\boldsymbol{b}})||^2 \\
&= {(\boldsymbol{\lambda} - \bar{\boldsymbol{\lambda}})}^{\prime} \bar{H}^{\prime}\bar{P}(\bar{H} \bar{\boldsymbol{x}} - \bar{\boldsymbol{b}}) \overset{\mathrm{(a)}}{=} 0.
\end{split}
\end{equation}
Note that the equality (a) holds because for every equilibrium, $\bar{H}^{\prime}\bar{P}(\bar{H} \bar{\boldsymbol{x}} - \bar{\boldsymbol{b}}) = \boldsymbol{0}$. Also, for every $(\bar{\boldsymbol{x}}, \bar{\boldsymbol{\lambda}}) \in \text{Equilibrium}(\boldsymbol{F})$,
\begin{equation}
\begin{split}
&\boldsymbol{F}(\boldsymbol{x}, \bar{\boldsymbol{\lambda}}) - \boldsymbol{F}(\bar{\boldsymbol{x}}, \bar{\boldsymbol{\lambda}}) = \\
&\boldsymbol{f}(\boldsymbol{x}) + {\bar{\boldsymbol{\lambda}}}^{\prime} \bar{H}^{\prime}\bar{P}(\bar{H} \boldsymbol{x} - \bar{\boldsymbol{b}}) + \frac{1}{2}||\bar{P}(\bar{H} \boldsymbol{x} - \bar{\boldsymbol{b}})||^2 \\
& -\boldsymbol{f}(\bar{\boldsymbol{x}}) - {\bar{\boldsymbol{\lambda}}}^{\prime} \bar{H}^{\prime}\bar{P}(\bar{H} \bar{\boldsymbol{x}} - \bar{\boldsymbol{b}}) - \frac{1}{2}||\bar{P}(\bar{H} \bar{\boldsymbol{x}} - \bar{\boldsymbol{b}})||^2. \\
\end{split}
\end{equation}
By Assumption \ref{assump:well_configured}, $\bar{H}^{\prime}\bar{P}(\bar{H} \bar{\boldsymbol{x}} - \bar{\boldsymbol{b}}) = \boldsymbol{0} \Rightarrow \bar{P}(\bar{H} \bar{\boldsymbol{x}} - \bar{\boldsymbol{b}}) = \boldsymbol{0}$. Thus,
\begin{equation}
\begin{split}
&\boldsymbol{F}(\boldsymbol{x}, \bar{\boldsymbol{\lambda}}) - \boldsymbol{F}(\bar{\boldsymbol{x}}, \bar{\boldsymbol{\lambda}}) = \\
&\boldsymbol{f}(\boldsymbol{x}) - \boldsymbol{f}(\bar{\boldsymbol{x}}) + {\bar{\boldsymbol{\lambda}}}^{\prime} \bar{H}^{\prime}\bar{P}\bar{H}(\boldsymbol{x} - \bar{ \boldsymbol{x}}) + \frac{1}{2}||\bar{P}(\bar{H} \boldsymbol{x} - \bar{\boldsymbol{b}})||^2 \\
&\geq \boldsymbol{f}(\boldsymbol{x}) - \boldsymbol{f}(\bar{\boldsymbol{x}}) + {\bar{\boldsymbol{\lambda}}}^{\prime} \bar{H}^{\prime}\bar{P}\bar{H}(\boldsymbol{x} - \bar{ \boldsymbol{x}}).
\end{split}
\end{equation}
By \eqref{eq:eq_F_define}, since for every equilibrium $\nabla \boldsymbol{f}(\bar{\boldsymbol{x}}) + \bar{H}^{\prime}\bar{P}\bar{H} \bar{\boldsymbol{\lambda}} = \boldsymbol{0}$ and $\bar{P}^{\prime}=\bar{P}$,
\begin{equation}
\begin{split}
\boldsymbol{F}(\boldsymbol{x}, \bar{\boldsymbol{\lambda}}) - \boldsymbol{F}(\bar{\boldsymbol{x}}, \bar{\boldsymbol{\lambda}}) &\geq \boldsymbol{f}(\boldsymbol{x}) - \boldsymbol{f}(\bar{\boldsymbol{x}}) + {\bar{\boldsymbol{\lambda}}}^{\prime} \bar{H}^{\prime}\bar{P}\bar{H}(\boldsymbol{x} - \bar{ \boldsymbol{x}}) \\
&= \boldsymbol{f}(\boldsymbol{x}) - \boldsymbol{f}(\bar{\boldsymbol{x}}) - {(\boldsymbol{x} - \bar{ \boldsymbol{x}})}^{\prime} \nabla \boldsymbol{f}(\bar{\boldsymbol{x}}) \geq 0.
\end{split}
\end{equation}
Note that the last inequality holds because for a convex function $\boldsymbol{f}$, $\boldsymbol{f}(\boldsymbol{y}) \geq \boldsymbol{f}(\boldsymbol{x}) + {(\boldsymbol{y} - \boldsymbol{x})}^{\prime} \nabla \boldsymbol{f}(\boldsymbol{x})$ for any $\boldsymbol{x}, \boldsymbol{y} \in \mathbb{R}^{mn}$.

Therefore, for every $(\bar{\boldsymbol{x}}, \bar{\boldsymbol{\lambda}}) \in \text{Equilibrium}(\boldsymbol{F})$, there holds
$\boldsymbol{F}(\bar{\boldsymbol{x}}, \boldsymbol{\lambda}) \leq \boldsymbol{F}(\bar{\boldsymbol{x}}, \bar{\boldsymbol{\lambda}}) \leq \boldsymbol{F}(\boldsymbol{x}, \bar{\boldsymbol{\lambda}})$. Thus by Definition \ref{definition:saddle_point}, every equilibrium of saddle-point dynamics \eqref{eq:algorithm_primal_dual_compact} is a global saddle point of the augmented Lagrangian $\boldsymbol{F}$, i.e. $\text{Equilibrium}(\boldsymbol{F}) \subseteq \text{Saddle}(\boldsymbol{F})$.
\end{proof}

Third, we need the following two lemmas to prove the convergence of saddle point dynamics, which are borrowed from \cite[Theorem 11.59]{rockafellar2009variational} and \cite[Corollary 4.2]{cherukuri2017saddle}, respectively.
\begin{lemma} \label{lemma:saddle_point_to_optimal_solution} \cite[Theorem 11.59]{rockafellar2009variational}
Suppose that $(\boldsymbol{x}^*, \boldsymbol{\lambda}^*)$ is a global saddle point of the augmented Lagrangian $\boldsymbol{F}: \mathbb{R}^{mn} \times \mathbb{R}^{mn} \mapsto \mathbb{R}$ in \eqref{eq:lagrange_define}. Then $\boldsymbol{x}^*$ is a globally optimal solution to the primal problem \eqref{eq:problem_equivalent}.
\end{lemma}

\begin{lemma} \label{lemma:global_asymptotic_stability_convex_concave} \cite[Corollary 4.2]{cherukuri2017saddle}
For $\boldsymbol{F}:\mathbb{R}^n \times \mathbb{R}^m \mapsto \mathbb{R}$ continuously differentiable and globally strictly convex-concave, \emph{Saddle}($\boldsymbol{F}$) is globally asymptotically stable under the saddle-point dynamics. The convergence of the trajectories is to a point.
\end{lemma}

Finally, the proof of Theorem \ref{theorem:saddle_point_theorem} is provided below.
\begin{proof}[\textbf{Proof of Theorem \ref{theorem:saddle_point_theorem}}]

By Lemma \ref{lemma:saddle_point}, every equilibrium of the dynamics \eqref{eq:algorithm_primal_dual_compact} is a global saddle point of the augmented Lagrangian $\boldsymbol{F}$. Then by Lemma \ref{lemma:saddle_point_to_optimal_solution}, every equilibrium is a globally optimal solution to the primal problem \eqref{eq:problem_equivalent}. The next step is to show that the proposed update rule \eqref{eq:algorithm_primal_dual_single} drives $(\boldsymbol{x}, \boldsymbol{\lambda})$ to $\text{Saddle}(\boldsymbol{F})$.

To show the asymptotic convergence, $\boldsymbol{f} = \sum_{i = 1}^m \boldsymbol{f}_i$ is assumed to be strictly convex in $\mathbb{R}^{mn}$. $\boldsymbol{f}$ is continuously differentiable since every $\boldsymbol{f}_i$ is continuously differentiable by assumption. By Definition \ref{definition:convex-concave}, $\boldsymbol{F}$ is globally strictly convex-concave because $\boldsymbol{f}$ is globally strictly convex and $\boldsymbol{F}$ is linear in its second argument $\boldsymbol{\lambda}$. Hence, $\text{Equilibrium}(\boldsymbol{F})$ is globally asymptotically stable, and the trajectories starting from any initial points given the update rule \eqref{eq:algorithm_primal_dual_single} converge to one point, i.e. one of the global optima of the equivalent problem \eqref{eq:problem_equivalent}.

Then with Lemma \ref{lemma:edge_agreement}, given any $\boldsymbol{x}(0)$ and the update rule \eqref{eq:algorithm_primal_dual_single}, the state $\boldsymbol{x}(t) \triangleq \text{col}\{ \boldsymbol{x}_1(t), \cdots, \boldsymbol{x}_m(t) \} \in \mathbb{R}^{mn}$ converges asymptotically to one of the globally optimal solutions $\boldsymbol{x}^*$ of \eqref{problem_interest}.
This completes the proof.
\end{proof}

Theorem \ref{theorem:saddle_point_theorem} requires the convexity of $\boldsymbol{f}_i$ to allow each $\boldsymbol{x}_i$ converges globally to $\boldsymbol{x}_i^*$. Without the convexity of each $\boldsymbol{f}_i$, one may only obtain locally asymptotic convergence.

\subsection{Special Case for Edge-Agreement Only}

To further highlight the connection between consensus and edge agreements, this subsection presents a special case for the edge-agreement-only problem. For this problem, a distributed algorithm without involving Lagrangian multipliers is given by
\begin{equation} \label{eq:update_only_edge}
\dot{\boldsymbol{x}}_i(t)= - \textstyle\sum_{j\in \mathcal{N}_i}P_{ij}(\boldsymbol{x}_i(t)-\boldsymbol{x}_j(t)-\bar{\boldsymbol{b}}_{ij}), \ \forall i \in \mathcal{V}.
\end{equation}

Although Ref. \cite{rai2023edge} has developed a discrete-time algorithm to achieve edge agreements, the above update \eqref{eq:update_only_edge} provides a continuous-time, distributed algorithm to solve the same problem.
Convergence of the proposed update \eqref{eq:update_only_edge} is given by the following corollary.

\begin{corollary}[\textbf{Edge-Agreement-Only}]\label{theorem:edge_only}
Under Assumption \ref{assum:existence}, Assumption \ref{assum:consistency}, and Assumption \ref{assump:well_configured}, the distributed update in continuous-time \eqref{eq:update_only_edge} drives each $\boldsymbol{x}_i(t)$ to converge exponentially fast to a constant vector $\boldsymbol{x}_i^*$, $\forall i \in \mathcal{V}$, which reach the edge agreement as defined in \eqref{eq:edge_agreement_require}.
\end{corollary}

\begin{proof}
Since $P_{ij}=P_{ji}$ and $\boldsymbol{b}_{ij}=-\boldsymbol{b}_{ji}$ because of Assumption \ref{assum:consistency}, then the distributed update in \eqref{eq:update_only_edge} can be written in the following compact form
\begin{equation} \label{eq:update_only_edge_compact}
\dot{\boldsymbol{x}} =- \bar{H}' \bar{P} ( \bar{H} \boldsymbol{x}-\bar{\boldsymbol{b}} ),
\end{equation}
where $\bar{P}$, $\bar{\boldsymbol{b}}$ and $\bar{H}$ are as defined in \eqref{eq:diag_mat_define}, \eqref{eq:edge_b_define} and \eqref{eq:H_bar_define}, respectively.
Define
\begin{equation} \label{eq:error_edge_only}
\boldsymbol{e}(t) \triangleq \bar{P} ( \bar{H} \boldsymbol{x}-\bar{\boldsymbol{b}}).
\end{equation}
With $\bar{P}^2=\bar{P}$ and \eqref{eq:update_only_edge_compact}, \eqref{eq:error_edge_only} leads to
\begin{equation} \label{eq:error_edge_only_dynamics}
\dot{\boldsymbol{e}} = - M \boldsymbol{e},
\end{equation}
where $M \triangleq \bar{P} \bar{H} \bar{H}' \bar{P}$.
Since $\bar{P}$ is symmetric, one notes right away that $M=(\bar{H}' \bar{P})' (\bar{H}' \bar{P})$ is positive semi-definite.
By Assumption \ref{assump:well_configured}, $\bar{H}' \bar{P} \boldsymbol{e}(t) \rightarrow \boldsymbol{0}$ exponentially fast. Then given the definition of $\boldsymbol{e}(t)$ in \eqref{eq:error_edge_only} and $\bar{P}^2 = \bar{P}$, $\bar{H}' \bar{P} (\bar{H} \boldsymbol{x}(t)-\bar{\boldsymbol{b}}) \rightarrow \boldsymbol{0}$ exponentially fast. It follows that there exists a constant vector $\boldsymbol{x}^*$ such that $\boldsymbol{x}(t) \rightarrow \boldsymbol{x}^*$ exponentially fast and $ \bar{H}' \bar{P} ( \bar{H} \boldsymbol{x}^*-\bar{\boldsymbol{b}} ) = \boldsymbol{0},$ which with Assumption \ref{assump:well_configured}  implies $\bar{P} ( \bar{H} \boldsymbol{x}^*-\bar{\boldsymbol{b}} )= \boldsymbol{0}$. From this with Assumption \ref{assum:existence} and Lemma \ref{lemma:edge_agreement}, one concludes that each $\boldsymbol{x}_i(t) \rightarrow \boldsymbol{x}_i^*$, $\forall i \in \mathcal{V}$, exponentially fast and all  $\boldsymbol{x}_i^*$ satisfy the edge agreement in \eqref{eq:edge_agreement_require}.
\end{proof}

\section{Simulations} \label{section:example}

This section includes two simulations to validate the proposed algorithms \eqref{eq:algorithm_primal_dual_single} and \eqref{eq:update_only_edge} for solving multi-agent optimization under edge agreements \eqref{problem_interest} and the edge-agreement-only problem \eqref{eq:edge_agreement_require}, respectively.

\subsection{Multi-Agent Formation Control}

This subsection applies the proposed distributed algorithm for edge-agreement-only \eqref{eq:update_only_edge} to a multi-agent formation control problem \cite{cao2011formation}. 
The multi-agent formation here is based on relative positions, which can be formulated as a group of edge agreements.
Suppose there are $m=4$ agents. Denote $\boldsymbol{x}_i \in \mathbb{R}^2$ as the position for agent $i$, which is initialized randomly.

Agents share and update their states (positions in this case) and should eventually achieve some desired edge agreements \eqref{example:desired_edge_agree} and hence a spatial formation in Fig. \ref{fig:desired_formation} shall be maintained.
The desired formation is shown in Fig. \ref{fig:desired_formation} with the edge set $\mathcal{E}=\{ (1,2), (2,3), (3,1), (3,4) \}$.
\begin{figure}[h]
\centering
\includegraphics[width=0.20\textwidth]{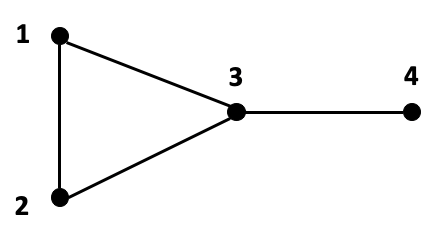}
\caption{Desired formation for four agents}
\label{fig:desired_formation}
\end{figure}
The desired formation can be represented as a group of edge agreements, i.e.
\begin{equation} \label{example:desired_edge_agree}
    A_{ij}(\boldsymbol{x}_{i} - \boldsymbol{x}_{j}) = \boldsymbol{b}_{ij}, \ \forall (i,j) \in \mathcal{E},
\end{equation}
where
\begin{equation*}
\begin{split}
    &A_{ij} = I_2, \ \forall (i,j) \in \mathcal{E}, \\
    &\boldsymbol{b}_{12} = -\boldsymbol{b}_{21} = \matt{0 & 3}^{\top}, \boldsymbol{b}_{23} = -\boldsymbol{b}_{32} = \matt{-2.6 & -1.5}^{\top}, \\
    &\boldsymbol{b}_{31} = -\boldsymbol{b}_{13} = \matt{2.6 & -1.5}^{\top}, \boldsymbol{b}_{34} = -\boldsymbol{b}_{43} = \matt{-3 & 0}^{\top}. \\
\end{split}
\end{equation*}
The oriented incidence matrix is
\begin{equation*}
    H = \matt{ 1 & -1 & 0 & 0 \\ 0 & 1 & -1 & 0 \\ 1 & 0 & -1 & 0 \\ 0 & 0 & 1 & -1}.
\end{equation*}

To validate Corollary \ref{theorem:edge_only}, each entry of $\boldsymbol{x}(0) \in \mathbb{R}^{8}$ is randomly sampled from a uniform distribution $[-10, 10]$.
To measure the closeness of all agents' states to the desired states satisfying edge agreements, introduce the following index
\begin{equation}
    V(t) = \frac{1}{2}\textstyle\sum_{(i,j) \in \mathcal{E}} ||A_{ij}(\boldsymbol{x}_i(t) - \boldsymbol{x}_j(t))-\boldsymbol{b}_{ij}||^2,
\end{equation}
where $V(t) \geq 0$ and $V(t)=0$ if and only if all the edge agreements in \eqref{eq:edge_agreement_require} are achieved. The simulation result is obtained by applying the update \eqref{eq:update_only_edge} with the {\tt{ode45}} solver of MATLAB. The process of agents achieving a formation is shown in Fig. \ref{fig:formation_result_only}. As shown in Fig. \ref{fig:v_plot}, the proposed algorithm \eqref{eq:update_only_edge} enables all agents to exponentially achieve the formation defined by the edge agreements in \eqref{example:desired_edge_agree}.

\begin{figure}[h]
\centering
\includegraphics[width=0.33\textwidth]{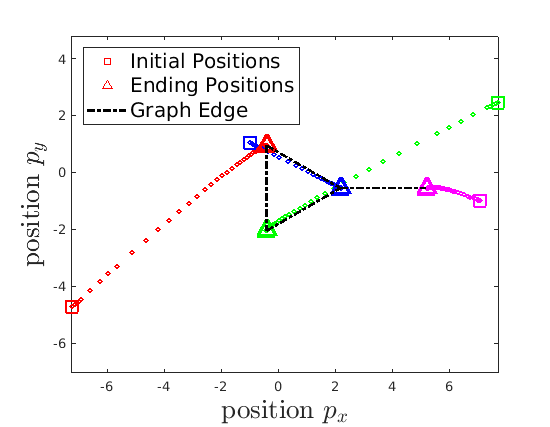}
\caption{Formation of four agents. Trajectories in red, green, blue, and magenta indicate agents 1 - 4, respectively. Squares are agents' initial positions. Triangles are agents' ending positions. Black dashed lines indicate graph edges.}
\label{fig:formation_result_only}
\end{figure}

\begin{figure}[h]
\centering
\includegraphics[width=0.33\textwidth]{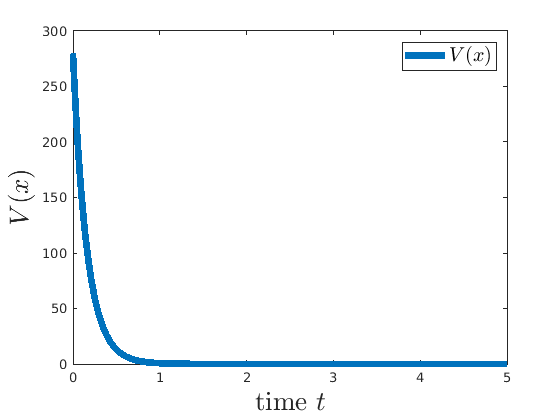}
\caption{Evolution of $V(t)$ for multi-agent formation control}
\label{fig:v_plot}
\end{figure}

\subsection{Multi-Agent Optimization with Formation Control}

This section applies the proposed distributed algorithm \eqref{eq:algorithm_primal_dual_single} to a multi-agent optimization problem with formation control.
This problem adopts the same setting about agents and network topology.
In addition, agents share and update their states and cooperatively minimize a global objective function by individually minimizing their local objective functions. Similarly, all the agents should eventually achieve some desired edge agreements \eqref{example:desired_edge_agree}, and hence a formation in Fig. \ref{fig:desired_formation} shall be maintained.

The local objective functions are defined as follows:
\begin{equation}
\begin{split}
&\boldsymbol{f}_1(\boldsymbol{x}_1) = ||\boldsymbol{x}_1||^2, \ \boldsymbol{f}_2(\boldsymbol{x}_2) = ||\boldsymbol{x}_2 - \matt{2 & 2}^{\top}||^2, \\
& \boldsymbol{f}_3(\boldsymbol{x}_3) = ||\boldsymbol{x}_2 + \matt{3 & 3}^{\top}||^2, \ \boldsymbol{f}_4(\boldsymbol{x}_4) = \textstyle\sum_{i=1}^2 e^{\boldsymbol{x}_4[i]},
\end{split}
\end{equation}
where $\boldsymbol{x}_4[i] \in \mathbb{R}$ denotes the $i$-th entry of vector $\boldsymbol{x}_4$.
$\boldsymbol{f}_i$ defines a specific point/region that each agent $i$ is desired to arrive at.
Note that $\boldsymbol{f}(\boldsymbol{x}) = \sum_{i=1}^m \boldsymbol{f}_i(\boldsymbol{x}_i)$ is strictly convex in $\boldsymbol{x}$.

To validate Theorem \ref{theorem:saddle_point_theorem}, each entry of $\boldsymbol{x}(0) \in \mathbb{R}^{8}$ is randomly sampled from a uniform distribution $[-10, 10]$. The global optimum $\boldsymbol{x}^* \in \mathbb{R}^8$ is obtained by solving the problem of interest \eqref{problem_interest} with MATLAB {\tt{fmincon}} solver in a centralized manner. Define the following function:
\begin{equation}
    W(t) = \frac{1}{2}\textstyle\sum_{i=1}^4 ||\boldsymbol{x}_i(t) - \boldsymbol{x}_i^*||^2,
\end{equation}
where $W(t)=0$ if and only if $\boldsymbol{x}_i(t) = \boldsymbol{x}_i^*$ for all $i \in \mathcal{V}$. The simulation result is obtained by applying the update \eqref{eq:algorithm_primal_dual_single} with the {\tt{ode45}} solver of MATLAB. The process of agents achieving a formation is shown in Fig. \ref{fig:formation_result}. As shown in Fig. \ref{fig:w_plot}, $\boldsymbol{x}(t)$ converges to the global optimum $\boldsymbol{x}^*$ asymptotically. Moreover, the constant slope of $W(t)$ on the logarithmic scale indicates the exponential convergence of the algorithm in this strictly convex function $\boldsymbol{f}(\boldsymbol{x})$.

\begin{figure}[h]
\centering
\includegraphics[width=0.33\textwidth]{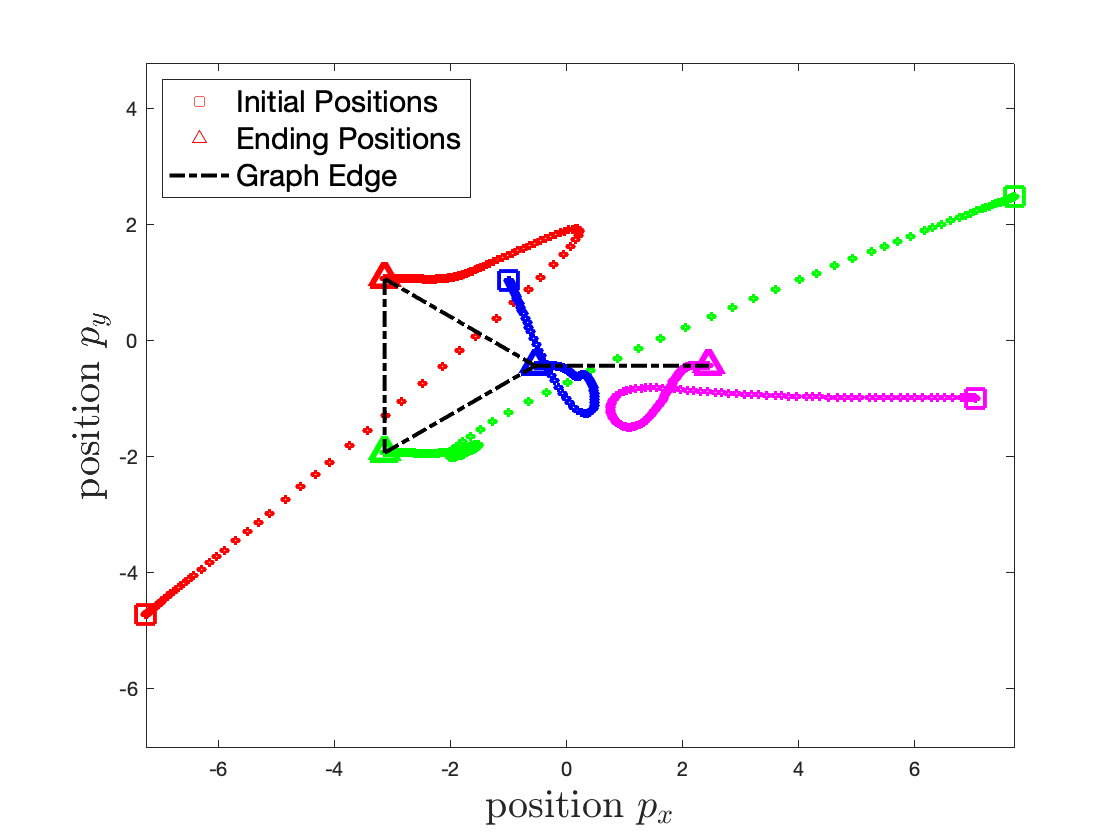}
\caption{Formation of four agents to minimize a global objective. Trajectories in red, green, blue, and magenta indicate agents 1 - 4, respectively. Squares are agents' initial positions. Triangles are agents' ending positions. Black dashed lines indicate graph edges.}
\label{fig:formation_result}
\end{figure}

\begin{figure}[h]
\centering
\includegraphics[width=0.33\textwidth]{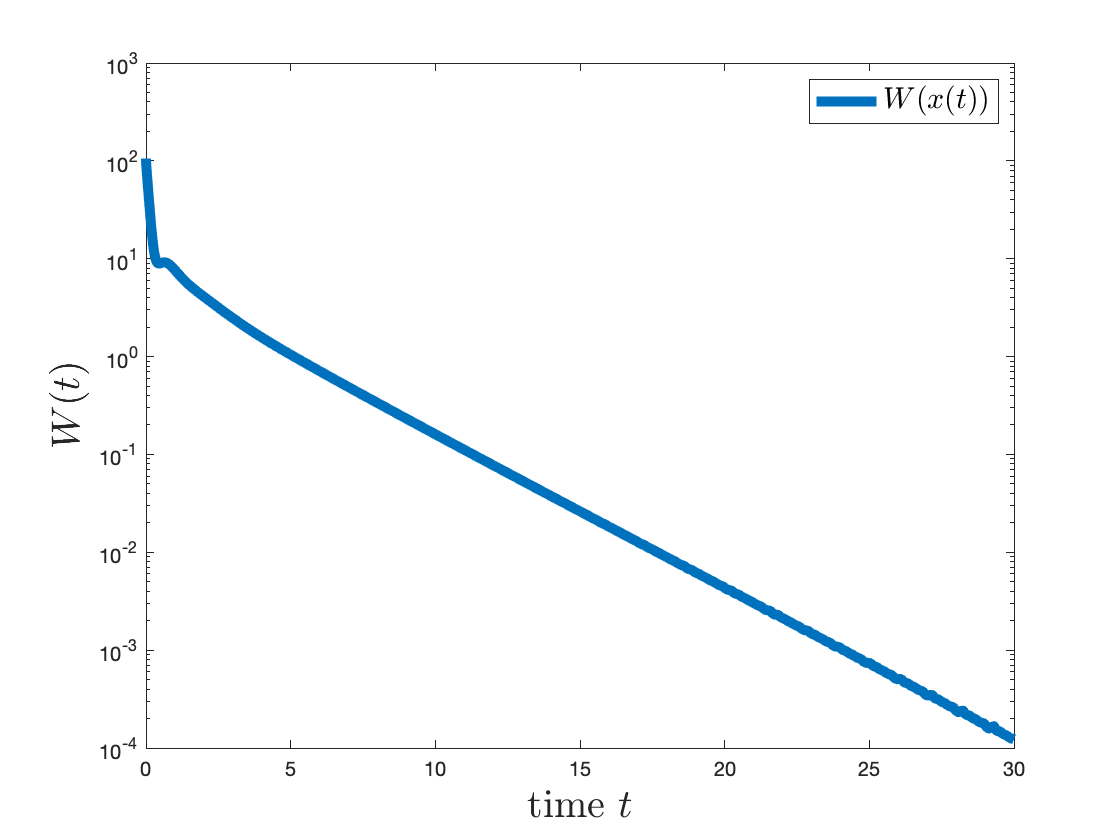}
\caption{Evolution of $W(t)$ on the logarithmic scale for multi-agent optimization with formation control}
\label{fig:w_plot}
\end{figure}

\section{Conclusion \& Discussion} \label{section:discussion}

This paper has introduced the concept of edge agreement, which is more general than consensus. Such edge agreements can enable more flexible coordination among multi-agent systems by enforcing a linear constraint to each pair of neighboring agents, rather than a global requirement of consensus in a multi-agent system. Based on the concept of edge agreements, the paper has developed a distributed algorithm to solve a new multi-agent optimization problem, which is a promising tool for multi-agent formation control.

From the similarity between edge agreement and consensus, one might transfer the existing consensus analysis techniques and generalize some useful analysis to edge agreement, such as convergence rate analysis.
Other future work includes the nonconvex optimization analysis, the generalization of the proposed algorithm to a discrete-time update with time-varying and directed networks, applying multi-agent optimization under edge agreements to distributed MPC problems \cite{stewart2010cooperative,fang2016cooperative,10155935}, and finding more applications of edge agreements.

\bibliographystyle{elsarticle-num}

\bibliography{reference}

\end{document}